\documentclass[11pt,a4paper]{article} 
\usepackage{a4wide}
\usepackage[latin1]{inputenc} 
\usepackage{algorithmic} 
\usepackage{amsmath}
\usepackage{amsfonts} 
\usepackage{amssymb} 
\usepackage{graphicx}
\usepackage{amsthm} 
\usepackage{mathtools} 
\usepackage[ruled]{algorithm2e}
\usepackage{enumerate} 
\usepackage{color} 
\usepackage{subfigure}
\usepackage{epigraph} 
\usepackage{authblk}
     
\mathtoolsset{showonlyrefs=true}

\newtheorem{definition}{Definition}[section]

\newtheorem{remark}{Remark}[section]
\newtheorem{lemma}{Lemma}[section]
\newtheorem{theorem}{Theorem}[section]
\newtheorem{fact}{Fact}[section]

\newtheorem{corollary}{Corollary}[section]
\newtheorem{alg}{Algorithm}[section]

\title{An Optimal Block Diagonal Preconditioner for Heterogeneous
  Saddle Point Problems in Phase Separation}

\author[1]{Pawan Kumar\thanks{kumar@zedat.fu-berlin.de}}
\affil[1]{Arnimallee 6}
\affil[1]{Department of Mathematics and Computer Science}
\affil[1]{Freie Universität Berlin}
\affil[1]{Berlin 14195, Germany}
 
\date{\today}

\begin{document}

\maketitle
 
\begin{abstract}
  The phase separation processes are typically modeled by Cahn-Hilliard
  equations. This equation was originally introduced to model phase separation
  in binary alloys; where phase stands for concentration of different components
  in alloy. When the binary alloy under preparation is subjected to a rapid
  reduction in temperature below a critical temperature, it has been
  experimentally observed that the concentration changes from a mixed state to a
  visibly distinct spatially separated two phase for binary alloy. This rapid
  reduction in the temperature, the so-called ``deep quench limit'', is modeled
  effectively by obstacle potential.

  The discretization of Cahn-Hilliard equation with obstacle potential leads to
  a block $2 \times 2$ {\em non-linear} system, where the $(1,1)$ block has a
  non-linear and non-smooth term. Recently a globally convergent Newton Schur
  method was proposed for the non-linear Schur complement corresponding to this
  non-linear system. The proposed method is similar to an inexact active set
  method in the sense that the active sets are first approximately identified by
  solving a quadratic obstacle problem corresponding to the $(1,1)$ block of the
  block $2 \times 2$ system, and later solving a reduced linear system by
  annihilating the rows and columns corresponding to identified active sets. For
  solving the quadratic obstacle problem, various optimal multigrid like methods
  have been proposed. In this paper, we study a non-standard norm that is
  equivalent to applying a block diagonal preconditioner to the reduced linear
  systems. Numerical experiments confirm the optimality of the solver and
  convergence independent of problem parameters on sufficiently fine mesh.
\end{abstract}

\section{Introduction} 
The Cahn-Hilliard equation was first proposed in 1958 by
Cahn and Hilliard \cite{Cahn1958} to study the phase separation
process in a binary alloy. Here the term phase stands for the
concentration of different components in the alloy. It has been
empirically observed that the concentration changes from the mixed
state to a visibly distinct spatially separated two phase state when
the alloy under preparation is subjected to a rapid cooling below a
critical temperature. This rapid reduction in the temperature the
so-called deep quench limit has been found to be modeled efficiently
by obstacle potential proposed by Oono and Puri \cite{Oono1987} in
1987 and by Blowey and Elliot \cite[p. 237, (1.14)]{Blowey1991}. The
phase separation has been noted to be highly non-linear (point
nonlinearity to be precise), and the obstacle potential emulates the
nonlinearity and non-smoothness that is empirically observed. However,
handling the non-smoothness as well as designing robust iterative
procedure has been the subject of much active research in last
decades. Assuming semi-implicit time discretizations \cite{Blowey1992}
to alleviate the time step restrictions, most of the proposed methods
essentially differ in the way the nonlinearity and non-smoothness are
handled. There are two main approaches to handle the non-smoothness:
regularization around the non-smooth region \cite{Bosch2014} or an
active set approach \cite{Graeser2009} i.e., identify the active sets
and solve a reduced problem which is linear, in addition to ensuring
the global convergence of the Newton method by proper damping
parameter. The non-linear problem corresponding to Cahn-Hilliard with
obstacle potential could be written as a non-linear system in block $2
\times 2$ matrix form as follows:
\begin{align}
  \begin{pmatrix}
    F & B^T \\
    B & -C
  \end{pmatrix}
  \begin{pmatrix}
    u^* \\ w^*
  \end{pmatrix}
  \ni
  \begin{pmatrix}
    f \\ g
  \end{pmatrix}, \quad u^*, w^* \in \mathbb{R}^n \label{eqn:saddle}
\end{align}
where $u^*, w^*$ are unknowns, $F = A + \partial I_K,$ where $I_K$ denotes the
indicator functional of the admissible set $K.$ The matrices $A,C$ are
essentially Laplacian with $A$ augmented by a non-local term reflecting mass
conservation, a necessary condition in Cahn-Hilliard model. Both nonlinearity
and non-smoothness are due to the presence of term $\partial I_K$ in $F.$ By
nonlinear Gaussian elimination of the $u$ variables, the system above could be
reduced to a nonlinear Schur complement system in $w$ variables
\cite{Graeser2009}, where the nonlinear Schur complement is given by $C -
BF^{-1}B^T.$ In \cite{Graeser2009}, a globally convergent Newton method is
proposed for this nonlinear Schur complement system which is interpreted as a
preconditioned Uzawa iteration. Note that $F(x)$ is a set valued mapping due to
the presence of set-valued operator $\partial I_K;$ to solve the inclusion $F(x)
\ni y$ corresponding to the quadratic obstacle problem, many methods have been
proposed such as projected block Gauss-Seidel \cite{Barrett2004}, monotone
multigrid method \cite{Kornhuber1994, Kornhuber1996, Mandel1984}, truncated
monotone multigrid \cite{Graeser2009b}, and truncated Newton multigrid
\cite{Graeser2009b}. See the excellent review article \cite{Graeser2009b} that
compares these methods. By annihilating the corresponding rows and columns that
belongs to the active sets identified by solving the obstacle problem, we obtain
a reduced linear system  as follows
\begin{align}
  \begin{pmatrix}
    \hat{A} & \hat{B}^T \\
    \hat{B} & -C
  \end{pmatrix}
  \begin{pmatrix}
    \hat{u} \\ \hat{w}
  \end{pmatrix}
  =
  \begin{pmatrix}
    \hat{f} \\ \hat{g}
  \end{pmatrix}, \quad \hat{u}, \hat{w} \in \mathbb{R}^n \label{eqn:red}
\end{align}
that correspond to new descent direction in the Uzawa iteration. 
The overall nonlinear iteration is performed in the sense of inexact
Uzawa, and the preconditioners are updated with next available active
sets.

In this paper our goal is to design effective preconditioner and hence an
iterative solver for \eqref{eqn:red} such that the convergence rate is
independent of problem parameters. In particular, we consider a block diagonal
preconditioner proposed in \cite{Bosch2014}; we adapt it to our linear
system, we prove properties relevant for iterative solvers, derive spectral
radius of the preconditioned operator, and show the effectiveness of the
preconditioner numerically compared to a Schur complement preconditioner
proposed recently for same model.

The rest of this paper is organized as follows. In Section \ref{sec:ch}, we
describe the Cahn-Hilliard model with obstacle potential, we discuss the time
and space discretizations and variational formulations. In Section
\ref{sec:solver}, we discuss briefly the solver for Cahn-Hilliard with obstacle
problem. In particular, we briefly discuss Nonsmooth Newton Schur method seen as
an Uzawa iteration, and the truncated Newton multigrid for the obstacle
problem. The preconditioners for the reduced linear systems are discussed in
Section \ref{sec:precon}. Finally in Section \ref{sec:numexp}, we shown
numerical experiments with the proposed preconditioner.

\section{\label{sec:not}Notations}
Let SPD and SPSD denote symmetric positive definite and symmetric positive semi
definite. Let $\kappa(M)$ denote the condition number of SPD matrix $M.$ For $x
\in \mathbb{R}, |x|$ denotes the absolute value of $x,$ whereas for any set
$\mathcal{K},$ $|\mathcal{K}|$ denotes the number of elements in $\mathcal{K}.$
Let $Id \in \mathbb{R}^{n \times n}$ denote the identity matrix. Let
$\mathbf{1}$ denote $[1,1,1,\dots,1].$ For a symmetric matrix $Z \in
\mathbb{R}^{n \times n},$ the eigenvalues are denoted and ordered as follows
\begin{align}
  \lambda_{1}(Z) \leq \lambda_{2}(Z) \leq \dots \leq \lambda_{n}(Z).
\end{align}

\section{\label{sec:ch} Cahn-Hilliard Problem with Obstacle Potential}
\subsection{The Model}
The Ginzburg-Landau (GL) energy functional which is given as follows
\begin{align}
  E(u) = \int_{\Omega} \frac{\epsilon}{2}|\nabla u|^2 +
  \psi(u) \, dx, \quad \Omega=(0,1) \times (0,1) \label{eqn:gle}
\end{align}
leads to Cahn-Hilliard equation under $H^{-1}$ gradient flow. Here the
constant $\epsilon$ relates to interfacial thickness and the obstacle
potential $\psi$ is given as follows:
\begin{align}
  \psi(u) = \psi_0(u) + I_{[-1, 1]}(u), \quad \mbox{where} \: \psi_0(u) =
  \frac{1}{2}(1-u^2).
\end{align}
Here the subscript $[-1,1]$ of indicator function $I$ above denotes the range of
values of $u.$ Let $u_1$ and $u_2$ be the concentration of the two components in
the binary alloy, then $u=u_1 - u_2,$ where $u_1,u_2 \in [0,1].$ Here
$I_{[-1,1]}(u)$ is defined as follows:
\begin{align}
  I_{[-1,1]} = \begin{cases} 0, \: \mbox{if} \: u(i) \in [-1,1], \\
    \infty, \: \mbox{otherwise}.
    \end{cases}
\end{align}
Moreover, $u_1+u_2$ is assumed to be conserved. 
We consider weak form of $H^{-1}$ gradient flow of $\epsilon E$ as follows 
\begin{align*}
  (\partial_t u, v)_{H^{-1}} = \epsilon (-\nabla E(u), v) 
  \iff ((-\Delta)^{-1} \partial u, v)_{L^2} = \epsilon \langle -\nabla E(u), v
  \rangle. 
\end{align*} 
And strong form reads 
\begin{align*}
  (-\nabla)^{-1}\partial_t u = -\epsilon \nabla E(u) 
  \iff \partial_t u = \epsilon (-\nabla)(-(\nabla E)(u)) = \Delta
   \epsilon \partial E(u).
\end{align*}
Now setting $w = \epsilon \partial E(u)$ above. From \eqref{eqn:gle} with
$\gamma=1$, we have
\begin{align*}
  \partial E(u) = -\epsilon \Delta u + \frac{1}{\epsilon} (\psi_0'(u) + \mu) 
  \implies w = \epsilon \partial E(u) = - \epsilon^2 \Delta u + \psi_0'(u) + \mu.
\end{align*} 
Putting everything together, the Cahn-Hilliard equation in PDE form with
inequality constraints obtained from GL energy \eqref{eqn:gle} reads:
\begin{align}
  \partial_t u &= \Delta w, \\
  w &= - \epsilon \Delta u + \psi^{'}_{0}(u) + \mu, \\
  \mu &\in \partial I_{[-1,1]}(u), \\
%  |u| &\leq 1, \\
  \frac{\partial u}{\partial n} &= \frac{\partial w}{\partial n} = 0
  \: \mbox{on} \, \partial \Omega.
\end{align}
The unknowns $u$ and $w$ are called order parameter and chemical potential
respectively.  For a given $\epsilon > 0,$ final time $T>0,$ and initial
condition $u_0 \in \mathcal{K}$ where
\begin{align} \mathcal{K}=\{ v \in H^1(\Omega) \, : \, |v| \leq
  1\}, \label{eqn:setK} \end{align} the equivalent initial value problem for
Cahn-Hilliard equation with obstacle potential interpreted as variational
inequality reads
\begin{align}
  \left \langle \frac{du}{dt}, v \right \rangle + (\nabla w, \nabla v)
  &= 0, \: \forall v \in H^1(\Omega), \label{eqn:ch31} \\
  \epsilon (\nabla u, \nabla(v-u)) - (u, v-u) &\geq (w,v-u), \: \forall v \in
  \mathcal{K}, \label{eqn:ch32}
\end{align}
where we use the notation $\langle \cdot, \cdot \rangle$ to denote the duality
pairing of $H^1(\Omega)$ and $H^1(\Omega)^{'}.$ Note that we used the fact that
$\psi^{'}_{0}(u) = -u$ in the second term on the left of inequality
\eqref{eqn:ch32} above. The inequalities \eqref{eqn:ch31} and \eqref{eqn:ch32}
are defined on constrained set $\mathcal{K},$ the variational inequality of
first kind is also equivalently represented on unconstrained set using indicator
functional \cite[p. 2]{Glowinski2008}. The existence and uniqueness of the
solution of \eqref{eqn:ch31}, \eqref{eqn:ch32} above has been established in
Blowey and Elliot \cite{Blowey1991}. We next consider an appropriate
discretization in time and space for the model.
\subsection{Time and space discretizations}\label{timeAndSpace}
We consider a fixed non-adaptive grid in time interval (0,T) and in space $\Omega$
defined in \eqref{eqn:gle}. The time step $\tau = T/N$ is kept uniform. We
consider the semi-implicit Euler discretization in time and finite element
discretization in space as in Barrett et. al. \cite{Barrett2004} with triangulation
$\mathcal{T}_h$ with the following spaces (as in \cite{Graeser2009}):
\begin{align}
  \mathcal{S}_h &= \left\{ v \in C(\overline{\Omega}) \, : v|_T ~\mbox{is
    linear} \quad \forall T \in T_h \right\}, \\
  \mathcal{P}_h &= \left\{ v \in L^2(\Omega) \, : \, v_T \,
  \mbox{is constant} \quad \forall T \in \mathcal{T} \in \mathcal{T}_h \right\}, \\
  \mathcal{K}_h &= \left\{ v \in \mathcal{P}_h \, : \, \left|\, v|_T \, \right| \leq 1 \quad \forall T
  \in \mathcal{T}_h \right\} = \mathcal{K} \cap \mathcal{S}_h \subset \mathcal{K},
\end{align}
which leads to the following discrete Cahn-Hilliard problem with obstacle
potential: \\\\
Find $u^k_h \in \mathcal{K}_h, w^k_h \in \mathcal{S}_h$ s.t. 
\begin{align}
  \langle u^k_h, v_h \rangle + \tau (\nabla w^k_h, \nabla v_h)
  &= \langle u^{k-1}_h, v_h \rangle, \: \forall v_h \in \mathcal{S}_h, \label{eqn:dche1} \\
  \epsilon (\nabla u^k_h, \nabla (v_h - u^k_h)) - \langle w^k_h, v_h - u^k_h
  \rangle &\geq \langle u^{k-1}_h, v_h - u^k_h \rangle, \: \forall v_h \in
  \mathcal{K}_h \label{eqn:dche2}
\end{align}
holds for each $k=1,\dots,N.$ The initial solution $u^0_h \in \mathcal{K}_h$ is
taken to be the discrete $L^2$ projection $\langle u^0_h, v_h \rangle = (u_0,
v_h), \forall v_h \in \mathcal{S}_h.$ 

Existence and uniqueness of the discrete Cahn-Hilliard equations has been
established in \cite{Blowey1992}. The discrete Cahn-Hilliard equation is
equivalent to the set valued saddle point block $2 \times 2$ nonlinear system
\eqref{eqn:saddle} with $F = A + \partial I_{\mathcal{K}_h}$ and
\begin{align}
  A = \epsilon (\langle \lambda_p, 1 \rangle \langle \lambda_p, 1 \rangle
  + (\nabla \lambda_p, \nabla \lambda_q))_{p,q \in \mathcal{N}_h}, \\
  B = (\langle \lambda_p, \lambda_q \rangle)_{p,q \in \mathcal{N}_h}, \: C =
  \tau ((\nabla \lambda_p, \nabla \lambda_q))_{p,q \in \mathcal{N}_h}. \label{eqn:stiff}
\end{align}
We write the above in more compact notations as follows
\begin{align}
  A = \epsilon(K+mm^T), \quad B = M, \quad C = \tau K, \label{eqn:short_not}
\end{align}
where $m = \langle \lambda_p, 1 \rangle,$ $M$ and $K$ are mass and stiffness
matrices respectively.
\section{\label{sec:solver}Iterative solver for Cahn-Hilliard with
  obstacle potential}
In \cite{Graeser2009}, a nonsmooth Newton Schur method is proposed which is also
interpreted as a preconditioned Uzawa iteration. For a given time step $k,$ the
Uzawa iteration reads:
\begin{align}
  u^{i,k} &= F^{-1} (f^k - B^T w^{i,k}), \label{eqn:uzawa1} \\
  w^{i+1,k} &= w^{i,k} + \rho^{i,k} \hat{S}^{-1}_{i,k}(Bu^{i,k} - Cw^{i,k} -
  g^k) \label{eqn:uzawa2}
\end{align}
for the saddle point problem \eqref{eqn:saddle}. Here $i$ denotes the $i^{th}$
Uzawa step, and $k$ denotes the $k^{th}$ time step. Here $f^k$ and $g^k$ are
defined as follows
\begin{align}
  \langle f, v_h\rangle = \langle u_h^{k-1}, v_h \rangle, \quad \langle g, v_h
  \rangle = - \langle u_h^{k-1}, v_h \rangle.
\end{align}
The time loop starts with an initial value for $w^{0,0}$ which can be taken
arbitrary as the method is globally convergent, and with the initial value
$u^{0,0}.$ The Uzawa iteration requires three main computations that we describe
below.
\subsection{Computing $u^{i,k}$}
The first step \eqref{eqn:uzawa1} corresponds to solving a quadratic obstacle
problem interpreted as a minimization problem as follows
\begin{align}
  u^{i,k} = \mbox{arg}~\underset{v \in K}{\mbox{min}} \left(\frac{1}{2}\langle
    Av, v \rangle - \langle f^k - B^T w^{i,k}, v \rangle \right).
\end{align}
As mentioned in the introduction, this problem has been extensively studied
during last decades \cite{Barrett2004, Graeser2009b, Kornhuber1994,
  Kornhuber1996}.
\subsubsection{Algebraic Monotone Multigrid for Obstacle Problem}
To solve the quadratic obstacle problem \eqref{eqn:uzawa1}, we use the monotone
multigrid method proposed in \cite{Kornhuber1994}. In Algorithm \ref{alg:mmg} we
describe an algebraic variant of the method. The algorithm performs one V-cycle
of multigrid; it takes $u^i$ from the previous iteration, and outputs the
improved solution $u^{i+1}.$ The initial set of interpolation operators are
constructed using aggregation based coarsening \cite{Kumar2014}.
\subsection{Computing $\hat{S}^{-1}_{i,k}(Bu^{i,k} - Cw^{i,k} - g^k)$}
The quantity $d^{i,k} = \hat{S}^{-1}_{i,k} (Bu^{i,k} - Cw^{i,k} - g^k)$ in
\eqref{eqn:uzawa2} is obtained as a solution of the following reduced linear
block $2 \times 2$ system:
\begin{align}
  \begin{pmatrix}
    \hat{A} & \hat{B}^T \\
    \hat{B} & -C
  \end{pmatrix}
  \begin{pmatrix}
    \tilde{u}^{i,k} \\ d^{i,k} 
  \end{pmatrix} =
  \begin{pmatrix}
    0 \\ g + Cw^{i,k} - Bu^{i,k} \label{eqn:rls}
  \end{pmatrix},
\end{align}
where 
\begin{align}
\hat{A} = TAT + \hat{T}, \quad \hat{B} = BT. \label{eqn:trunc_mass_stiff}
\end{align}
Here truncation matrices $T$ and $\hat{T}$ are defined as follows:
\begin{align} \label{eqn:TandThat}
  T = \mbox{diag}
  \begin{pmatrix*}[l]
    0, \quad \mbox{if}~u^{i,k}(j) \in \{-1, 1\} \\
    1, \quad \mbox{otherwise}
  \end{pmatrix*}, \quad
  \hat{T} = \mbox{diag}
  \begin{pmatrix*}[l]
    1, \quad \mbox{if}~T_{jj}=0 \\
    0, \quad \mbox{otherwise}
  \end{pmatrix*}, \quad j=1,\dots,|N_h|,
\end{align}
where $u^{i,k}(j)$ is the $j$th component of $u^{i,k},$ and $T_{jj}$ is the
$j$th diagonal entry of $T.$ In words, $\hat{A}$ is the matrix obtained from $A$
by replacing the $i$th row and $i$th column by the unit vector $e_i$
corresponding to the active sets identified by diagonal entries of $T.$
Similarly, $\hat{B}$ is the matrix obtained from $B$ by annihilating rows, and
$\hat{B}^T$ is the matrix obtained from $B$ by annihilating columns.
Rewriting untruncated version of \eqref{eqn:rls} in simpler notation as follows
\begin{align}
 \left( \begin{array}{cc}
         \epsilon \bar{K} & M \\
         M & -\tau K
        \end{array} \right) \left( \begin{array}{c}
                            x \\ y
                           \end{array} \right) = \left( \begin{array}{c}
                                                         0 \\ b
                                                        \end{array} 
                                                        \right),
\end{align}
where $\bar{K} = K + mm^T.$ By a change of variable $y' = y/\epsilon,$ we obtain
\begin{align}
 \left( \begin{array}{cc}
         \bar{K} & M \\
         M  & -\eta K
        \end{array}\right) \left( \begin{array}{c}
                            x \\ y'
                           \end{array} \right) = \left( \begin{array}{c}
                                                         0 \\ b
                                                        \end{array} 
                                                        \right).
\end{align}
Furthermore, we modify the $(2,2)$ term of the system matrix above as follows:
\begin{align}
 - \eta K = - \eta
K - \eta mm^T + \eta mm^T = - \eta \bar{K} + (\eta^{1/2}m) (\eta^{1/2}m^T) = - \eta \bar{K} +
\tilde{m} \tilde{m}^T. 
\end{align}
Now the untruncated system may be rewritten as
\begin{align}
 \widetilde{\mathcal{A}} = \begin{pmatrix}
                \bar{K} & M \\
		M & -\eta \bar{K}  
               \end{pmatrix} + \hat{m} \hat{m}^T = \mathcal{A} + \hat{m}\hat{m}^T               \label{eqn:simplified}
\end{align}
where $\hat{m} = [0, \tilde{m}^T]$ is a rank one term with proper extension by
zero. Now we are in a position to use Sherman-Woodbury inversion for matrix plus
rank-one term. In the following, we shall concern ourselves in developing 
efficient preconditioners to solve with $\mathcal{A}$ in \eqref{eqn:simplified}.

\subsection{Computing step length $\rho^{i,k}$}
The step length $\rho^{i,k}$ is computed using a bisection method. We refer 
the interested reader to \cite{Graeser2011}[p. 88].
 
\begin{algorithm}
  \begin{algorithmic}[1]
    \REQUIRE Let $V_1 \subset V_1 \subset V_2 \cdots V_m$ and let $r_m, b_m \in
    V_m,$ \REQUIRE $u^i, i>0$ solution from previous cycle or $u^0$ the initial
    solution \STATE Compute residual: \: $r_m = b_m - A_m u^i$ \STATE Compute
    defect obstacles: \begin{align} \begin{cases}
        \underline{\delta}_m = \underline{\psi} - u^i \\
        \overline{\delta}_m = \overline{\psi} - u^i
      \end{cases} \end{align} \FOR{$\ell=m,\cdots,2$} \STATE Projected
    Gauss-Seidel Solve
    \begin{align}
      (D_\ell + L_\ell + \partial I_{\mathcal{K}^\ell})v_\ell = r_\ell
    \end{align}
    where,
    \begin{align}
      \mathcal{K}^\ell = \left\{\, v \in \mathbb{R}^{n_\ell} \,
        \middle| \: \underline{\delta}_{\ell} \leq v \leq
        \overline{\delta}_{\ell} \, \right\}
    \end{align}
    using Algorithm \ref{alg:pgs}.
    % as follows:
    % \begin{align}
    %   v_{\ell} \leftarrow PGS()
    % \end{align}
    \STATE Update
    \begin{align}
      \begin{cases} r_{\ell} := r_{\ell} - A_{\ell} v_{\ell} \\
        \underline{\delta}_{\ell - 1} := \underline{\delta}_{\ell} - v_{\ell} \\
        \overline{\delta}_{\ell - 1} := \overline{\delta}_{\ell} -
        v_{\ell}
      \end{cases}
    \end{align}
    \STATE Restrict and compute new obstacle
    \begin{align}
      \begin{cases}
        r_{\ell-1} = P^T_{\ell-1}r_{\ell}\\
        (\underline{\delta}_{\ell-1})_i := \max \left\{
          (\underline{\delta}_{\ell-1})_j \: \middle| \:
          (P_{\ell-1})_{ji} \neq 0 \right\}, \: i=1,\dots,n_{\ell} -1 \\
        (\overline{\delta}_{\ell-1})_i := \min \left\{
          (\overline{\delta}_{\ell-1})_j \: \middle| \:
          (P_{\ell-1})_{ji} \neq 0 \right\}, \: i=1,\dots,n_{\ell} -1
      \end{cases}
    \end{align}
    \ENDFOR
    \STATE Solve
    \begin{align}
      (D_1 + L_1 + \partial I_{\mathcal{D}^1})v_1 = r_1
    \end{align}
    \FOR{$\ell = 2, \cdots, m$} \STATE Add corrections
    \begin{align}
      v_{\ell} := v_{\ell} + P_{\ell-1}v_{\ell-1}
    \end{align}
    \ENDFOR
    \STATE Compute
    \begin{align}
      u^{i+1} = u^i + v_m
    \end{align}
    \ENSURE improved solution $u^{i+1}$
  \end{algorithmic}
  \caption{\label{alg:mmg}Monotone Multigrid (MMG) V cycle}
\end{algorithm}

\begin{algorithm}
  \begin{algorithmic}[1]
    \REQUIRE $A \in \mathbb{R}^{n_{\ell} \times n_{\ell}}, \, b,
    \underline{\psi}, \overline{\psi} \in \mathbb{R}^{n_{\ell}},$
    current iterate $x^i \in \mathbb{R}^{n_{\ell}}$ \ENSURE new
    iterate $x^{i+1} \in \mathbb{R}^{n_{\ell}}$ \STATE Compute
    residual: \[r := b - Ax^i \] \STATE Compute defect
    obstacles: \[\underline{\psi} := \underline{\psi} - x^i \] \[
    \overline{\psi} := \overline{\psi} - x^i \] \FOR{$i=1:n_{\ell}$}
    \FOR{$j=1:i$} \STATE Compute $y_i$
    \begin{align}
      y_i = \begin{cases}
        \max \left(\min \left((r_i - A_{ij}y_j)/A_{ii}, \: \overline{\psi}_i \right), \: \underline{\psi}_i \right), \quad \mbox{if} \: A_{ii} \neq 0, \\
        0, \quad \mbox{otherwise}
      \end{cases}
    \end{align}
    \ENDFOR
    \ENDFOR
    \STATE $x^{i+1} = x^i+y$
  \end{algorithmic}
  \caption{\label{alg:pgs} $x^{i+1}$ $\leftarrow$ PGS($x^i, A,
    \underline{\psi}, \overline{\psi}, b$)}
\end{algorithm}

\subsection{\label{sec:mixedfem} Mixed Finite Element Formulation of Reduced
  Linear System}
In section \ref{timeAndSpace}, we already discussed the finite element
discretization of the reduced linear system. To fit our problem into the
Zulehner's approach, we rewrite the PDE corresponding to the reduced linear
system, the corresponding continuous weak form in product space using mixed
bilinear forms.  For corresponding weak formulations of \eqref{eqn:rls}. We
first write the corresponding partial differential equations as follows:
\begin{equation}
\begin{aligned}
  \epsilon \Delta u + \lambda &= f \quad \mbox{in} \, \Omega_I \subset \Omega,  \\
  u - \tau \Delta \lambda &= g \quad \mbox{in} \, \Omega,  \\
  \frac{\partial u}{\partial n} = \frac{\partial \lambda}{\partial n} &= 0 \quad \mbox{on} \, \partial \Omega,  \\
  u &= 0 \quad \mbox{on} \ \partial \Omega_I \setminus \partial
  \Omega. \label{eqn:original}
\end{aligned}
\end{equation}
To this end, we choose suitable Hilbert spaces for trial (i.e. weak solution)
and test spaces as follows
\begin{align}
 \hat{V} &= \{v \in H^1(\Omega) : v|_{\Omega_A}=0 \}, \quad Q = H^1_0(\Omega), \label{eqn:V} 
 %V^0 &= V \cap H^1_0(\Omega). \label{eqn:V2} 
\end{align} 
where $\Omega_A = \Omega \setminus \Omega_I.$
The weak form of the partial differential equations \eqref{eqn:original} corresponding to \eqref{eqn:rls} is written as follows:\\\\
% The continuous weak variational form for problem \eqref{eqn:original} now reads:\\\\
Find $(u, \lambda) \in \hat{V} \times H^1(\Omega):$
\begin{align}
  \hat{a}(u, v) + \hat{b}(v, \lambda) &= f(v) \quad \mbox{for all} \: v \in \hat{V}, \label{eqn:weak1}\\
  \hat{b}(u, q) - c(\lambda, q) &= g(q) \quad \mbox{for all} \: q \in
  Q, \label{eqn:weak2}
\end{align}
where
\begin{equation}
\begin{aligned} \label{eqn:abc}
 \hat{a}(u,v) &= \epsilon \left( (\nabla u, \nabla v) + \int_{\Omega} u \int_{\Omega} v \right)
        = \epsilon \left((\nabla u, \nabla v) + \langle u, \mathbf{1} \rangle \langle v, \mathbf{1}\rangle \right) , \\ 
 \hat{b}(v,\lambda) &= (v, \lambda), \quad
 c(\lambda,q) = \tau(\nabla \lambda, \nabla q).
\end{aligned}
\end{equation}
We immediately observe the following trivial properties for the system
\eqref{eqn:weak1}-\eqref{eqn:weak2}.
\begin{theorem}[Properties of bilinear forms]\label{thm:properties}
 There holds
 \begin{enumerate}
  \item \label{item:1} $\hat{a}(\cdot, \cdot)$ is symmetric and coercive
  \item \label{item:2} $c(\cdot, \cdot)$ is symmetric and semi-coercive
  \item \label{item:3} $\hat{a}(\cdot, \cdot),$ $\hat{b}(\cdot, \cdot),$ and $c(\cdot,
    \cdot)$ are bounded
 \end{enumerate}
\end{theorem}
\begin{proof}
  \eqref{item:1} and \eqref{item:2} follows from Poincar\`e
  inequality. Boundedness of $\hat{b}(\cdot, \cdot)$ follows from Cauchy-Schwarz
  inequality, and boundedness of $\hat{a}(\cdot, \cdot)$ and $c(\cdot, \cdot)$ follows
  from Cauchy-Schwarz inequality followed by inverse inequality.
\end{proof}

The mixed variational problem above can also be written as a variational form on product spaces: \\\\
Find $x \in \hat{V} \times Q:$
\begin{align}
 \hat{\mathcal{B}} (x,y) = \mathcal{F}(y) \quad \mbox{for all}~ y \in V \times H^1(\Omega) \label{eqn:mixed_bilin}
\end{align}
where $\mathcal{B}$ and $\mathcal{F}$ are defined as follows
\begin{align}
\hat{\mathcal{B}} (z,y) = \hat{a}(w,v) + \hat{b}(v,r) + \hat{b}(w,q) - c(r,q), \quad
\mathcal{F}(y) = f(v) + g(q)        
\end{align}
for $y=(v,q) \in \hat{V} \times Q$ and $z = (w,r) \in \hat{V} \times Q.$ 
The corresponding bilinear form for the untruncated system is given as follows
\begin{align}
\mathcal{B} (z,y) = a(w,v) + b(v,r) + b(w,q) - c(r,q), \quad
\mathcal{F}(y) = f(v) + g(q)        
\end{align}
for $y=(v,q) \in V \times Q$ and $z = (w,r) \in V \times Q,$ where $V = H^1(\Omega).$ 
In the rest of this paper, we shall consider norms
also proposed in \cite{Zulehner2011} as follows
\begin{align}
  ((v,q), (w,r))_{\hat{X}} = (v,w)_{\hat{V}} + (q,r)_Q, \label{eqn:normX}
\end{align}
where $(\cdot, \cdot)_{\hat{V}}$ and $(\cdot, \cdot)_Q$ are inner products of
Hilbert spaces $\hat{V}$ and $Q,$ respectively.  We will see shortly that such
norms lead to block diagonal preconditioners.  The boundedness condition for the
mixed problem for truncated and untruncated problem is given as follows
\begin{align}
  \sup_{0 \neq z \in \hat{X}} \sup_{0 \neq y \in \hat{X}}
  \frac{\hat{\mathcal{B}}(z,y)}{\|z\|_{\hat{X}} \|y\|_{\hat{X}}} \leq \sup_{0
    \neq z \in X} \sup_{0 \neq y \in X} \frac{\mathcal{B}(z,y)}{\|z\|_X \|y\|_X}
  \leq \bar{c}_x < \infty. \label{eqn:bound}
\end{align}
However, for well-posedness, following well known Babuska-Brezzi condition needs
to be satisfied
\begin{align}
  \inf_{0 \neq z \in \hat{X}} \sup_{0 \neq y \in \hat{X}}
  \frac{\mathcal{\hat{B}}(z,y)}{\|z\|_{\hat{X}} \|y\|_{\hat{X}}} \geq \inf_{0
    \neq z \in X} \sup_{0 \neq y \in X} \frac{\mathcal{B}(z,y)}{\|z\|_X \|y\|_X}
  \geq \underbar{c}_x > 0. \label{eqn:bb}
\end{align}  
Here \eqref{eqn:bound} and \eqref{eqn:bb} are consequence of the fact that
$\hat{X} \subset X$ as $\hat{V} \subset V.$ We shall provide equivalent
conditions for \eqref{eqn:bound} and \eqref{eqn:bb} that are easier to check,
but more importantly it leads to optimal norms.  But first we need to introduce
some notations for operators corresponding to bilinear forms.  It is easy to see
that $\hat{V} \times H^1(\Omega)$ is a Hilbert space itself as $\hat{V}$ and $H^1(\Omega)$
are themselves Hilbert spaces. It is convenient to associate linear operators
for the bilinear forms $a,b,$ and $c$ as follows
\begin{equation}
\begin{aligned} \label{eqn:ops}
 \langle \hat{A}w, v \rangle &= \hat{a}(w,v), \quad \hat{A} \in L(\hat{V}, \hat{V}^*), \\
 \langle \hat{B}w, q \rangle &= \hat{b}(w,q), \quad \hat{B} \in L(\hat{V}, Q^*), \\
 \langle Cr, q \rangle &= c(w,v), \quad C \in L(Q, Q^*), \\
 \langle \hat{B}^*r, v \rangle &= \langle \hat{B}v, r \rangle, \quad \hat{B}^* \in L(Q, \hat{V}^*).
\end{aligned}
\end{equation}
Consequently $\mathcal{B}$ and $\mathcal{F}$ in operator notation are given as
follows
\begin{align}
 \hat{\mathcal{A}} = \begin{pmatrix}
                \hat{A} & \hat{B}^* \\
		\hat{B} & -C  
               \end{pmatrix}, \quad \mathcal{F} = 
               \begin{pmatrix}
                f \\ g
               \end{pmatrix}, \quad x = 
               \begin{pmatrix}
                u \\ p
               \end{pmatrix}.               
\end{align}
The problem \eqref{eqn:mixed_bilin} is now given in operator notation as follows
\begin{align}
 \hat{\mathcal{A}}x = \mathcal{F}. \label{eqn:AxF}
\end{align}
The corresponding untruncated problem is given as follows
\begin{align}
 \mathcal{A}x = \mathcal{F}, \label{eqn:AxF2}
\end{align}
where $\mathcal{A}$ consists of untruncated matrix $A$ in place of $\hat{A}$ and 
$B$ instead of $\hat{B}.$

In \cite{Zulehner2011}, starting from the abstract theory on Hilbert spaces that
lead to representation of isometries, a preconditioner is proposed; it is based
on non-standard norms or isometries that correspond to block diagonal
preconditioner of the following form
\begin{align}
  \mathcal{B} =
  \begin{pmatrix}
    \mathcal{I}_V & \\
    & \mathcal{I}_Q
  \end{pmatrix}.
\end{align}

\subsection{\label{sec:norms} Choice of norm: a brief introduction to Zulehner's idea}
In the discrete case, we unavoidably introduce an additional parameter, i.e, the
mesh size $h,$ in addition to the problem parameters $\tau$ and $\epsilon.$ Our
goal is look for norms that are independent of all these parameters. The content
of this section and the notations are inspired from \cite{Zulehner2011}.
% , but introduce further modification
% and approximation of these norms that suits our problem. Moreover, such
% approximations are motivated by the quest for iterative methods (as direct
% methods are unfeasable) that are applied cheaply for example via existing
% solvers such as Multigrid, for which retaining essential properties such as
% positive definiteness or $M-$matrix property is essential. 

Before we move further, we introduce some notations.  The duality pairing
$\langle \cdot, \cdot \rangle_H$ on $H^* \times H$ is defined as follows
\begin{align*}
  \langle \ell, x \rangle_H = \ell(x) \quad \mbox{for all}~\ell \in H^*, \: x
  \in H.
\end{align*}
Let $\mathcal{I}_H:H \rightarrow H^*$ be an isometric isomorphism defined as
follows
\begin{align*}
  \langle \mathcal{I}_Hx, y \rangle = (x,y)_H.
\end{align*}
The inverse $\mathcal{R}_H = \mathcal{I}^{-1}$ is Riesz-isomorphism, by which
functionals in $H^*$ can be identified with elements in $H$ and we have
\begin{align*}
  \langle \ell, x \rangle = (\mathcal{R}_H\ell, x)_H.
\end{align*}

We already
chose the norm \eqref{eqn:normX}, we now look for explicit representation of
isometries or norms in finite dimension.  For this norm, we briefly describe how
the norms are derived. The main ingredient is the following theorem.
\begin{theorem}[Zulehner 2011 \cite{Zulehner2011}]
  If there are constants $\underline{\gamma}_v, \overline{\gamma}_v,
  \underline{\gamma}_q, \overline{\gamma}_q > 0$ such that
 \begin{align}
   \underline{\gamma}_v \|w\|^2_V \leq a(w,w) + \| Bw \|^2_{Q^*} \leq
   \overline{\gamma}_v \| w\|^2_V, \quad \forall w \in V \label{eqn:infsup1}
 \end{align}
 and
 \begin{align}
   \underline{\gamma}_q \| r \|^2_Q \leq c(r,r) + \| B^*r \|^2_{V^*} \leq
   \overline{\gamma}_q \| r \|^2_Q, \quad \forall r \in Q \label{eqn:infsup2}
 \end{align}
then  
\begin{align} 
  \underline{c}_x \| z \|_X \leq \| \mathcal{A} z \|_{X^*} \leq \overline{c}_x
  \| z \|_X, \quad \forall z \in X \label{eqn:infsupA}
\end{align}
is satisfied with constants $\underline{c}_x, \overline{c}_x >0$ that depend
only on $\underline{\gamma}_v, \overline{\gamma}_v, \underline{\gamma}_q,
\overline{\gamma}_q.$ And, vice versa, if the estimates \eqref{eqn:infsupA} are
satisfied with constants $\underline{c}_x, \overline{c}_x > 0,$ then the
estimates \eqref{eqn:infsup1} and \eqref{eqn:infsup2} are satisfied.
\end{theorem}
In view of \eqref{eqn:bound} and \eqref{eqn:bb} and recalling $\hat{X} \subset
X,$ the following bounds hold for truncated system
\begin{align}
  \| \hat{A}\hat{z} \|_{\hat{X}^*} \geq \| Az \|_{X^*}, \quad \| \hat{A}\hat{z} \|_{\hat{X}^*} \leq \| Az \|_{X^*},
  \quad \forall \hat{z} \in \hat{X} \subset X, z \in X.
\end{align}

In \cite{Zulehner2011}, the terms $\| Bw \|^2_{Q^*}$ and $\| B^*r \|^2_{V^*}$ in
\eqref{eqn:infsup1} and \eqref{eqn:infsup2} respectively are defined as follows:
  \begin{align*}
    \|B w\|^2_{Q^*} = \langle B^*\mathcal{I}^{-1}_QBw,w \rangle, \quad
    \|B^*r\|^2_{V^*} = \langle B\mathcal{I}^{-1}_VB^*r, r \rangle.
  \end{align*}
  Then \eqref{eqn:infsup1} and \eqref{eqn:infsup2} are equivalently written as
  follows
\begin{align*}
  \underline{\gamma}_v \langle \mathcal{I}_V w, w \rangle & \leq \langle (A + B^*
  \mathcal{I}^{-1}_Q B)w, w \rangle \leq
  \bar{\gamma}_v \langle \mathcal{I}_V w, w \rangle \quad \mbox{for all}~w \in V, \\
  \underline{\gamma}_q \langle \mathcal{I}_Q r, r \rangle & \leq \langle (C + B
  \mathcal{I}^{-1}_V B^*)r, r \rangle \leq \bar{\gamma}_q \langle \mathcal{I}_Q
  r, r \rangle \quad \mbox{for all}~r \in Q.
\end{align*}
In short, in new notation $\sim$ meaning ``spectrally similar'', we obtain the
following equivalent conditions for isometries
\begin{align*}
  \mathcal{I}_V                                           & \sim A + B^*\mathcal{I}^{-1}_Q B \quad \mbox{and} \quad \mathcal{I}_Q \sim C + B \mathcal{I}^{-1}_V B^* \\
  \iff \mathcal{I}_V & \sim A + B^* (C + B \mathcal{I}^{-1}_V B^*)^{-1} B \quad
  \mbox{and} \quad \mathcal{I}_Q \sim C + B
  \mathcal{I}^{-1}_V B^*                                                                                                                                            \\
  \iff \mathcal{I}_Q & \sim C + B (A + B^* \mathcal{I}^{-1}_Q B)^{-1} B \quad
  \mbox{and} \quad \mathcal{I}_V \sim A + B^* \mathcal{I}^{-1}_Q B
\end{align*}
Let $M$ and $N$ be any SPD matrices, consequently, they define inner products
and a Hilbert space structure in $\mathbb{R}^n.$ The intermediate
Hilbert spaces between $M$ and $N$ are given as follows
\begin{align*}
  [M,N]_\theta=M^{1/2}(M^{-1/2}NM^{-1/2})^\theta M^{1/2}, \quad \theta \in
  [0,1].
\end{align*}
Continuing from above, in the case when $A$ and $C$ are non singular, the more
generic form of the norms are given by the following lemma
\begin{lemma}
  \label{I_V_and_I_Q}
  Let $A,C$ be nonsingular. Then 
  \begin{align}
    \mathcal{I}_V = A + [A, B^TC^{-1}B]_\theta, \quad \mathcal{I}_Q = C + [C,
    BA^{-1}B^T]_{1-\theta}, \quad \theta \in [0,1].
  \end{align}
\end{lemma}
\begin{proof}
  See \cite{Zulehner2011}[p. 547-548].
\end{proof}
Before we propose preconditioners, we shall need some properties of the $(1,1)$
block of $\mathcal{A},$ and that for the negative Schur complement $S = C +
\hat{B} \hat{A}^{-1}\hat{B}^T$ in \eqref{eqn:AxF}. These properties are used to
prove some bounds and to suggest approximation of norms.

\subsubsection{Properties of truncated $(1,1)$ block and Schur complement}
An important property that we shall need shortly when analyzing preconditioners
is that the eigenvalues of the truncated matrix is bounded from above and below
by the eigenvalues of the untruncated matrix. In this subsection, we also assume
that the grid is uniform. 

A result that we need later is the following.
\begin{lemma}
  \label{spd}
  $A, A + B$ is SPD.
\end{lemma}
\begin{proof}
  We have $A = \epsilon(K + mm^T), \epsilon > 0.$ $A$ is SPSP except on the
  vector $1 = [1,1,1,\dots,1]^T$ which is in the kernel of $A$ but $(mm^T1,
  1)>0.$ Also, since $B = M$ is SPD the proof follows.
\end{proof}

\begin{lemma}[Permutation preserves eigenvalues]
  \label{lem:perm_eig}
  Let $P \in \mathbb{Z}^{n \times n}$ be a permutation matrix, then $P^T \hat{A}
  P$ and $\hat{A}$ are similar.
\end{lemma}
\begin{proof}
  $P$ being a permutation matrix, $P^TP = Id,$ hence the proof.
\end{proof}
\begin{lemma}[Poincare separation theorem for eigenvalues]
  \label{lem:poincare}
  Let $Z \in \mathbb{R}^{n \times n}$ be any symmetric matrix with eigenvalues
  $\lambda_1 \leq \lambda_2 \leq \cdots \lambda_n,$ and let $P$ be a
  semi-orthogonal $n \times k$ matrix such that $P^TP = Id \in \mathbb{R}^{k
    \times k}.$ Then the eigenvalues $\mu_1 \leq \mu_2 \cdots \mu_{n-k+i}$ of
  $P^TZP$ are separated by the eigenvalues of $Z$ as follows
  \begin{align}
    \lambda_i \leq \mu_i \leq \lambda_{n-k+i}.
  \end{align}
\end{lemma}  
\begin{proof}  
  The theorem is proved in \cite[p. 337]{Rao1998}.
\end{proof}
\begin{lemma}[Eigenvalues of the truncated matrix]
  \label{thm:eig_bound}
  Let $\lambda_1 \leq \lambda_2 \dots \leq \lambda_n$ be the eigenvalues of $A,$
  and let $\hat{\lambda}_1 \leq \hat{\lambda}_2 \dots \leq \hat{\lambda}_n$ be
  the eigenvalues of truncated matrix $\hat{A}.$ Let $k=\sum_{i=1}^n T(i,i)$ be
  the number of untruncated rows in $\hat{A}.$ Let $\hat{\lambda}_{n_1} \leq
  \hat{\lambda}_{n_2} \dots \hat{\lambda}_{n_k}$ be the eigenvalues of $\hat{A}$
  due to addition of $\hat{T}.$ Then the following holds
  \begin{align}
    \lambda_i \leq \hat{\lambda}_{n_i} \leq \lambda_{n-k+i}.
  \end{align}
\end{lemma}
\begin{proof}
  The proof shall follow by application of Poincare separation theorem, to this
  end, we need to reformulate our problem. Let $P$ be a permutation matrix that
  renumbers the rows such that the truncated rows are numbered first, then we
  have
\begin{align}
  P^T\hat{A}P = \begin{pmatrix}
    I & \\
      & R^TP^T\hat{A}PR
  \end{pmatrix},
\end{align}
where $R \in \mathbb{R}^{n \times k}$ is the restriction operator defined as
follows
\begin{align}
  R =
  \begin{pmatrix}
    \begin{pmatrix*}[c]
      0 & 0 & \dots & 0 \\
      \vdots & \hdots & \hdots & 0 \\
      0 & 0 & \dots & 0
    \end{pmatrix*}_{n-k \times k} \\    
    \begin{pmatrix*}[c]
      1 & 0 & \dots & 0 \\
      0 & 1 & \dots & 0 \\
      \vdots & \ddots & \dots & 0 \\
      0 & 0 & \dots & 1
    \end{pmatrix*}_{k \times k}
  \end{pmatrix}.
\end{align}
Clearly $R^T R = Id \in \mathbb{R}^{k \times k}.$ From Lemma \ref{lem:perm_eig}, $P^T
\hat{A} P$ and $\hat{A}$ are similar. From Lemma \ref{lem:poincare}, theorem
follows.
\end{proof}
\begin{corollary}
  From Theorem \ref{thm:eig_bound}, we have $$\lambda_{\mbox{min}}(\hat{A}) \geq
  \lambda_{\mbox{min}}(A)>0,$$ 
$$\lambda_{\mbox{max}}(\hat{A}) \leq
  \lambda_{\mbox{max}}(A),$$
hence $\hat{A}$ is SPD.
  Moreover, $cond(\hat{A}) \leq cond(A).$
\end{corollary}
We know that the matrix $M$ is SPD, and $K$ is a SPSD.  In the following we
observe the properties of truncated matrices obtained from these.
\begin{definition}
  Let $G(A)=(V,E)$ be the adjacency graph of a matrix $A \in \mathbb{R}^{N
    \times N}$.  The matrix $A$ is called irreducible if any vertex $i \in V$ is
  connected to any vertex $j \in V$. Otherwise, $A$ is called reducible.
\end{definition}
\begin{definition} \label{def:mmatrix1} A matrix $A \in \mathbb{R}^{N \times N}$
  is called an $M-$matrix if it satisfies the following three properties:
  % \begin{enumerate}
  % \item $a_{ii}>0$ for $i=1,\dots,N$
  % \item $a_{ij}\le 0$ for $i \neq j$, $i,j=1,\dots,N$
  % \item $A$ is non-singular and $A^{-1} \ge 0$
  % \end{enumerate}
  $a_{ii}>0$ for $i=1,\dots,N,$
  $a_{ij}\le 0$ for $i \neq j$, $i,j=1,\dots,N,$ and
  $A$ is non-singular and $A^{-1} \ge 0.$
\end{definition}
\begin{definition}
  A square matrix $A$ is strictly diagonally dominant if the following holds
  \begin{align}
  |a_{ii}| > \sum_{j \neq i}|a_{ij}|, i=1,\dots,N, \label{eqn:diag_dom}
  \end{align}
  and it is called irreducibly diagonally dominant if $A$ is irreducible and the
  following holds
  \begin{align}
    |a_{ii}| \geq \sum_{j \neq i}|a_{ij}|, i=1,\dots,N, \label{eqn:irr_diag_dom} 
  \end{align}
  where strict inequality holds for at least one $i$.
\end{definition}
A simpler criteria for $M-$matrix property is given by the following theorem.
\begin{lemma} \label{thm:mmatrix} If the coefficient matrix $A$ is strictly or
  irreducibly diagonally dominant and satisfies the following conditions
  \begin{enumerate}
  \item $a_{ii}>0$ for $i=1,\dots,N$
  \item $a_{ij}\le 0$ for $i \neq j$, $i,j=1,\dots,N$
  \end{enumerate} then $A$ is an $M-$matrix.
\end{lemma}
 
\begin{remark}
  Note that $K$ is not $M-$matrix because $K \cdot \mathbf{1}=0,$ hence,
  \eqref{eqn:irr_diag_dom} is not satisfied. Moreover, mass matrix $M$ has
  positive off-diagonal entries, hence, it is not an $M-$matrix either. Alternatively,
  item 3. of Definition \ref{def:mmatrix1} is not satisfied.
\end{remark}
Let $\mathcal{N}^{\bullet}_h = \{i \: : \: T(i,i) = 0\}.$ 
\begin{lemma}
  Let $|\mathcal{N}^{\bullet}_h|\geq 1,$ then $\hat{K}, \, P^T \hat{K} P,$ and $R^T
  P^T \hat{K} P R$ are $M-$matrices.
\end{lemma}
\begin{proof}
  Since we have $|\mathcal{N}^{\bullet}_h|\geq 1,$ for all rows in truncated set
  $\mathcal{N}^{\bullet}_h,$ we have strict diagonal dominance
  \begin{align}
    \hat{k}_{ii} = 1 = |\hat{k}_{ii}| = 0 > \sum_{j \neq i}\hat{k}_{ij}, \quad
    \forall i \in \mathcal{N}^{\bullet}_h, \: j=1,\dots,
    |\mathcal{N}_h|, \label{eqn:strict_diag_dom}
  \end{align}
  where as, for rows corresponding to untruncated set $\mathcal{N}_h \setminus
  \mathcal{N}^{\bullet}_h,$
  \begin{align}
    \hat{k}_{ii} =  k_{ii} = |\hat{k}_{ii}| \geq \sum_{j \neq
      i} k_{ij} \geq \sum_{j \neq i}\hat{k}_{ij}, \quad \forall i
    \in \mathcal{N}_h \setminus \mathcal{N}^{\bullet}_h, \: j=1,\dots,
    |\mathcal{N}_h|. \label{eqn:dd}
  \end{align}
  Moreover, we have 
  \begin{align} \label{eqn:entries}
    \hat{k}_{ij} = 
    \begin{cases}(\mbox{when}~i=j)
      \begin{cases}
        1, \quad \forall i \in \mathcal{N}^{\bullet}_h, \\
        k_{ii} > 0, \quad \forall i \in \mathcal{N}_h \setminus
        \mathcal{N}^{\bullet}_h,
      \end{cases}\\
      (\mbox{when}~i \neq j) \quad k_{ij} < 0, \quad \forall i\in\mathcal{N}_h.
    \end{cases}
  \end{align}
  The sufficient conditions of Lemma \ref{thm:mmatrix} are now satisfied: from
  \eqref{eqn:strict_diag_dom} and \eqref{eqn:dd}, we conclude that $\hat{K}$ is
  irreducibly diagonaly dominant, and \eqref{eqn:entries} satisfies items 1. and
  2. of Lemma \ref{thm:mmatrix}. Hence $\hat{K}$ is an $M-$matrix.
  $P^T\hat{K}P$ being the symmetric permutation of rows and columns of $\hat{K}$
  above reasoning holds. Lastly, $R^T P^T \hat{K} P R$ being a principle
  submatrix of $P^T \hat{K} P$ is also an $M-$matrix, see proof in
  \cite{Horn1991}[p. 114].
\end{proof}
\begin{remark}
  From \eqref{eqn:short_not} and \eqref{eqn:trunc_mass_stiff}, the $(1,1)$ block
  $\hat{A} = \epsilon(\hat{K} + \hat{m}\hat{m}^T),$ where $\hat{K} = TKT +
  \hat{T}, \hat{m} = Tm.$ To solve with $\hat{A},$ we use the Sherman-Woodbury
  formula
\begin{align}
  \hat{A}^+ = (\hat{K} + \tilde{m}\tilde{m}^T)^{+} = \hat{K}^{+} -
  \frac{\hat{K}^{+}\tilde{m}\tilde{m}^T \hat{K}^{+}}{1 +
    \tilde{m}^T\hat{K}^{+} \tilde{m}}.
\end{align}
Here $\hat{K}^+$ denotes pseudo-inverse of $\hat{K},$ however, in our case
$\hat{K}$ is a non-singular $M-$matrix, see Definition \ref{def:mmatrix1}. Since
$\hat{K}^+$ is an $M-$matrix algebraic multigrid or incomplete Cholesky (which
is as
stable as exact Cholesky factorization, \cite{Meijerink1977}[Theorem 3.2] ) may
be used as a preconditioner to solve with $\hat{K}$ inexactly.
\end{remark}   
We provide a slightly different proof then in \cite{Graeser2009}. 
%Our proof is based on the fact that $M$ has positive entries.
\begin{theorem}
  \label{schur_spd}
  The negative Schur complement $S = C + \hat{B}\hat{A}^{-1}\hat{B}^T$ is
  non-singular, in particular, SPD if and only if $|\mathcal{N}_h^{\bullet}|>0.$
\end{theorem}
\begin{proof}
  If $|\mathcal{N}_h^{\bullet}|=0,$ then $\hat{B}$ is zero matrix, consequently
  $S = C$ is singular.  For other implication, we recall that $\hat{B}^T =
  \hat{M}^T = -TM$ where $T$ is defined in \eqref{eqn:TandThat}. The
  $(i,j)^{th}$ entry of element mass matrix is given as follows
  \begin{align}
    M^K_{ij} = \int_K \phi_i \phi_j dx = \frac{1}{12} (1 + \delta_{ij}
    |K|) \label{eqn:element_mass}
    \quad i,j=1,2,3,
  \end{align}
  where $\delta_{ij}$ is the Kronecker symbol, that is, 1 if $i=j,$ and 0 if $i
  \neq j.$ Here $\phi_1, \phi_2,$ and $\phi_3$ are hat functions on triangular
  element $K$ with local numbering and $|K|$ is the area of triangle element
  $K.$ From \eqref{eqn:element_mass}, it is easy to see that
  \begin{align} \label{eqn:element_mass}
    M^K = \frac{1}{12} \begin{pmatrix}
        2 & 1 & 1 \\
        1 & 2 & 1 \\
        1 & 1 & 2 
        \end{pmatrix}.
  \end{align}
  Evidently, entries of global mass matrix $M = \sum_K M^K$ are also all
  positive, hence all entries of truncated mass matrix $\hat{M}$ remain
  non-negative. In particular, due to our hypothesis $|\mathcal{N}^{\bullet}|>0,$
  there are at least one untruncated column, hence, at least few positive
  entries. Consequently, $\mathbf{1}$ is neither in kernel of $M$ nor in the
  kernel of $\hat{M},$ in particular, $\mathbf{1}^T \hat{M}^T \mathbf{1} > 0.$
  The proof of the theorem then follows since $C$ is SPD except on $\mathbf{1}$
  for which $\hat{B}^T \mathbf{1}$ is non-zero, and the fact that $\hat{A}$ is
  SPD yields
  \begin{align}
    \left \langle \hat{B}\hat{A}^{-1}\hat{B}^T \mathbf{1}, \mathbf{1} \right
    \rangle = \left \langle \hat{A}^{-1}(\hat{B}^T \mathbf{1}),
      (\hat{B}^T\mathbf{1}) \right \rangle = \left \langle
      \hat{A}^{-1}(-\hat{M}^T \mathbf{1}), (-\hat{M}^T\mathbf{1}) \right \rangle
    > 0.
  \end{align}
\end{proof}
\subsubsection{\label{sec:precII} Preconditioner I:} 
\begin{align}  
  \label{eqn:assume}
  B^T = B, \quad C = \eta A, \, \eta \neq 0.
\end{align}
Moreover $A$ hence $C$ are non-singular.  Specifically for $\theta=1/2$, and
using \eqref{eqn:assume} and Lemma \ref{I_V_and_I_Q} above \vspace{-0.2cm} yields
\begin{align*}
  \mathcal{I}_V = A + \eta^{-1/2}[A,
  BA^{-1}B]_{1/2}, \quad \mathcal{I}_Q = C +
  \eta^{1/2}[A, BA^{-1}B]_{1/2}.
\end{align*}
But $[A, BA^{-1}B]_{1/2} = B,$ thus further
simplification yields
\begin{align}
  \mathcal{I}_V = A + \eta^{-1/2}B, \quad \mathcal{I}_Q = C +
  \eta^{1/2}B. \label{eqn:pre2}
\end{align} 
Choice of $\theta=0,1$ brings back Schur Complements. For large problems, it
won't be feasible to solve with $\mathcal{I}_V$ and $\mathcal{I}_Q$ in
\eqref{eqn:pre2} exactly, or not even up to double precision using prohibitively
expensive direct methods such as QR or LU factorizations \cite{Golub1996}. 
\begin{remark}
For
existence and subsequent application of fast inexact solvers for $\mathcal{I}_V$ and
$\mathcal{I}_Q,$ an important property to look for
is $M-$matrix property, but unfortunately this property is lost in
\eqref{eqn:pre2}, consequently, the diagonal dominance of $\mathcal{I}_V$ or
$\mathcal{I}_Q$ may be lost for certain values of $\eta.$ To sketch the proof
for $I_Q$, we observe that
\begin{align}
  A^K_{ij} & = \left( \int_K \nabla \phi_i \cdot \nabla \phi_j dx
    + \int_K \phi_i dx
    \int_K \phi_j dx \right), \quad i,j=1,2,3, \\
                 & = (b_ib_j + c_ic_j) \int_K dx + mm^T = (b_ib_j + c_ic_j) |K| + mm^T,
  \quad i,j=1,2,3.
\end{align}
Reusing the definition of element mass matrix in \eqref{eqn:element_mass}, we
have
\begin{align}
  \eta A^K_{ij} + \eta^{1/2}B^K_{ij} = \eta A^K_{ij} +
  \eta^{1/2}M^K_{ij} = \eta (b_ib_j + c_ic_j)|K| +
  \eta^{1/2}\frac{1}{12}(1 + \delta_{ij}|K|) + mm^T.
\end{align}
It is not hard to see that for certain values of $\eta,$ diagonal dominance
property \eqref{eqn:diag_dom} is lost. Similarly, diagonal dominance property is
lost for $\mathcal{I}_Q.$ To retain the M-matrix property, it is advisable to lump 
the mass matrix. We proved earlier that the truncated matrix $\hat{K}$ is M-matrix if 
there is at least one truncated node, in that case $\hat{\mathcal{I}_Q}$ can be made 
M-matrix.
\end{remark}

% \begin{remark}
%   Since $B = -M,$ we choose
%   \begin{align}
%     \mathcal{I}_V = A - \eta^{-1/2}B^T, \quad \mathcal{I}_Q = C -
%     \eta^{1/2}B. \label{eqn:pre22}
%   \end{align}  
%   Note that the norms $\mathcal{I}_V$ and $\mathcal{I}_Q$ in \eqref{eqn:pre2}
%   are uneffected by the choice of sign of $B$ thus our choice above
%   \eqref{eqn:pre22} is justified.  Since $A$ is SPD (Lemma \ref{spd}) and $-B=M$
%   is positive definite, $\mathcal{I}_V$ is positive definite. By similar
%   argument $\mathcal{I}_Q$ is also negative definite. Aggregation based
%   algebraic multigrid \cite{kumar2014,Notay2010} may be used to solve with
%   $\mathcal{I}_V$ and $\mathcal{I}_Q$ (because $-\mathcal{I}_Q$ is SPD).
% \end{remark}

The following remark relates the eigenvalues of $\mathcal{A}$ to the eigenvalue 
of $A$ by a change of variable
\begin{remark}
 To this end we rewrite the system
as follows
\begin{align}
  \begin{pmatrix}
    (i\eta^{1/2}) \hat{A} & \hat{B}^T \\
    \hat{B}               & (i\eta^{1/2})\hat{A}
  \end{pmatrix} =   \begin{pmatrix}
    \tilde{A}             & \hat{B}   \\
    \hat{B}               & \tilde{A}
  \end{pmatrix}. \label{eqn:cmplx_mat}
\end{align}
Let $(\lambda, u)$ be an eigenpair of $\tilde{A}+\hat{B}$, then $(u^T, u^T)^T$
is an eigenvector of \eqref{eqn:cmplx_mat}. Similarly, let $(\mu, v)$ be an
eigenpair of $\tilde{A}-\hat{B}$, then $(v^T, -v^T)^T$ is also an eigenvector of
\eqref{eqn:cmplx_mat}. This implies that eigenvalues of \eqref{eqn:cmplx_mat}
are union of eigenvalues of $\tilde{A}+\hat{B}$ and $\tilde{A}-\hat{B}.$ 
We notice that the eigenvalues come in pairs with real part of
eigenvalues in each pair having opposite sign.  
\end{remark}
The following theorem estimates the spectral radius of the preconditioned
operator.
\begin{lemma}
  \label{spec_radius}
  The spectral radius of the preconditioned operator is given as follows
  \begin{align}
    \rho(\left[  
      \begin{array}{cc}
        \bar{K} + \eta^{-1/2}M & 0 \\
        0 & \eta \bar{K} + \eta^{1/2}M
      \end{array}
    \right]^{-1} \left[ 
      \begin{array}{cc}
        \bar{K} & M \\
        M & -\eta \bar{K}  
      \end{array}
    \right]) \leq \frac{1}{(1 + \eta C)} + \frac{Ch^4}{(1 + \eta)} 
  \end{align}
\end{lemma}
\begin{proof}
  Consider the generalized eigenvalue problem
  \begin{align}
    \left[\begin{array}{cc}
      \bar{K} & M \\
      M & -\eta \bar{K}
    \end{array}\right]
  \left[\begin{array}{c}
    v \\
    u
  \end{array}\right] = \lambda
\left[\begin{array}{cc}
  \bar{K} + \eta^{-1/2}M & 0 \\
  0 & \eta \bar{K} + \eta^{1/2}M
\end{array}\right]  \left[\begin{array}{c}
    v \\
    u
  \end{array}\right]
  \end{align}
Taking inner product on both sides with $[v^T, u^T],$ we have
  \begin{align}
   [v^T, u^T] \left[\begin{array}{cc}
      \bar{K} & M \\
      M & -\eta \bar{K}
    \end{array}\right]
  \left[\begin{array}{c}
    v \\
    u
  \end{array}\right] &= \lambda [v^T, u^T]
\left[\begin{array}{cc}
  \bar{K} + \eta^{-1/2}M & 0 \\
  0 & \eta \bar{K} + \eta^{1/2}M
\end{array}\right]  \left[\begin{array}{c}
    v \\
    u
  \end{array}\right] \\
   [v^T, u^T] \left[\begin{array}{c}
      \bar{K}v + Mu \\
      Mv -\eta \bar{K}u
    \end{array}\right] &= \lambda [v^T, u^T]
\left[\begin{array}{c}
  \bar{K}v + \eta^{-1/2}Mv   \\
   \eta \bar{K}u + \eta^{1/2}Mu
\end{array}\right] \\
v^T\bar{K}v + v^TMu + u^TMv - \eta u^T \bar{K} u &= \lambda [v^T \bar{K}v + \eta^{-1/2}v^TMv + \eta
u^T \bar{K} u + \eta^{1/2}u^T M u] \\
v^T \bar{K} v + 2Re(Mu,v) - \eta u^T \bar{K} u &= \lambda [v^T \bar{K} v + \eta^{-1/2}v^TMv + \eta
u^T \bar{K} u + \eta^{1/2}u^T M u] 
  \end{align}
Taking absolute value, we get
\begin{align}
\left| v^T \bar{K} v + 2Re(Mu,v) + \eta u^T \bar{K} u \right| &= \left|\lambda \left[ v^T \bar{K} v + \eta^{-1/2}v^TMv + \eta
u^T \bar{K} u + \eta^{1/2}u^T M u \right] \right| \\
\|v\|^2_{\bar{K}} + \eta \|u\|^2_{\bar{K}} + 2 \left| Re(Mu,v) \right|  &= |\lambda| \left[ \|v\|^2_{\bar{K}} + \eta^{-1/2} \|
v \|_M^2 + \eta \|u\|_{\bar{K}}^2 + \eta^{1/2}\|u\|_M^2 \right]  \\
|\lambda| &= \frac{\|v\|^2_{\bar{K}} + \eta \|u\|^2_{\bar{K}} + 2 \left| Re(Mu,v) \right| }{\left[ \|v\|^2_{\bar{K}} + \eta^{-1/2} \|
v \|_M^2 + \eta \|u\|_{\bar{K}}^2 + \eta^{1/2}\|u\|_M^2 \right]}
\end{align}
We note that $M$ being mass matrix $\|M\| \leq Ch^2$ for any finite dimensional norm.
We now estimate $| Re (Mu, v) |$ 
\begin{align}
  |(Mu, v)| &\leq C_1 h^2 |v|_1 |u|_1 \quad (\mbox{Cauchy-Schwarz ineq.})\\
  &\leq C_1 h^2 \|v\|_{\bar{K}} \|u\|_M \quad (\mbox{Equivalence of norms in finite dimension})\\
  &\leq C_1 C_2 C_3 h^2 \left( \frac{\|v\|^2_{\bar{K}} + \|u\|^2_M}{2} \right) \quad
  (\mbox{Young's inequality})
\end{align}
Thus we have
\begin{align}
  2|(Mu,v)| &\leq 2 \frac{Ch^4}{2}(\|v\|^2_{\bar{K}} + \|u\|^2_M) \quad (\mbox{set}~C = C_1C_2C_3) \\
  &= Ch^2 \|v\|^2_{\bar{K}} + \|u\|^2_M 
\end{align}
With this, the estimate for $|\lambda|$ becomes
\begin{align}
  |\lambda| &\leq \frac{\|v\|^2_{\bar{K}} + \eta \|u\|^2_{\bar{K}} + Ch^4
    \left( \|v\|^2_{\bar{K}} + \|u\|^2_M \right) }{ \|v\|^2_{\bar{K}} + \eta^{-1/2}
      \|v\|_M^2 + \eta \|u\|_{\bar{K}}^2 + \eta^{1/2}\|u\|_M^2 } \\
  &\leq \frac{ \|v\|^2_{\bar{K}} + \eta \|u\|^2_{\bar{K}} }{ \|v\|^2_{\bar{K}} +
      \eta^{-1/2} \|v\|_M^2 + \eta \|u\|_{\bar{K}}^2 + \eta^{1/2}\|u\|_M^2} +
  \frac{Ch^4 \left( \|v\|^2_{\bar{K}} + \|u\|^2_M \right) }{ \|v\|^2_{\bar{K}} + \eta^{-1/2}
      \|v\|_M^2 + \eta \|u\|_{\bar{K}}^2 + \eta^{1/2}\|u\|_M^2 } \\
  &= \frac{1}{\frac{ \|v\|^2_{\bar{K}} +
      \eta^{-1/2} \|v\|_M^2 + \eta \|u\|_{\bar{K}}^2 + \eta^{1/2}\|u\|_M^2 } {\|v\|^2_{\bar{K}} + \eta \|u\|^2_{\bar{K}} }} +
  \frac{Ch^4 }{\frac{ \|v\|^2_{\bar{K}} + \eta^{-1/2}
      \|v\|_M^2 + \eta \|u\|_{\bar{K}}^2 + \eta^{1/2}\|u\|_M^2 }{ \|v\|^2_{\bar{K}} + \|u\|^2_M }} \\
  &\leq \frac{1}{1 + \eta \left(\frac{\|v\|_M^2 + \eta \|u\|_M^2}{\|v\|_{\bar{K}}^2 +
      \eta \|u\|_{\bar{K}}^2}\right)} + \frac{Ch^4 }{\frac{ \|v\|^2_{\bar{K}} + \eta^{-1/2}
      \|v\|_M^2 + \eta \|u\|_{\bar{K}}^2 + \eta^{1/2}\|u\|_M^2 }{
      \|v\|^2_{\bar{K}} + \|u\|^2_M }} \\
  &\leq \frac{1}{1 + \eta \left(\frac{\|v\|_M^2 + \eta \|u\|_M^2}{\|v\|_{\bar{K}}^2 +
      \eta \|u\|_{\bar{K}}^2}\right)} + \frac{Ch^4}{ 1 + \left(\frac{\eta \|u \|_{\bar{K}}^2
  + \eta^{1/2}\|u\|_M^2}{\|v\|_{\bar{K}}^2 + \|u\|_M^2} \right)} \quad
(\mbox{since}~\eta^{-1/2} > 1) \\
  &\leq \frac{1}{1 + \eta \left(\frac{\|v\|_M^2 + \eta \|u\|_M^2}{\|v\|_{\bar{K}}^2 +
      \eta \|u\|_{\bar{K}}^2}\right) } + \frac{Ch^4}{ 1 + \eta \left(\frac{\|u \|_{\bar{K}}^2
  + \|u\|_M^2}{\|v\|_{\bar{K}}^2 + \|u\|_M^2} \right)} \quad
(\mbox{since}~\eta = \mbox{min}(\eta, \eta^{1/2})) \\
\end{align}
Using equivalence of norms in finite dimension
\begin{align}
  \|u\|_M \leq C_4 \|u\|_{\bar{K}}, \quad \|v\|_M \leq C_5\|v\|_{\bar{K}}
\end{align}
we have
\begin{align}
  |\lambda|^2 &\leq \frac{1}{(1 + \eta C_6)} + \frac{Ch^4}{(1 + \eta)} 
\end{align}
\end{proof}
The following Lemma shows that condition number is of the order one.
\begin{theorem}
  \label{cond}
  The asymptotic condition number is given as follows
  \begin{align}
       \kappa \left(\left[  
      \begin{array}{cc}
        \bar{K} + \eta^{-1/2}M & 0 \\
        0 & \eta \bar{K} + \eta^{1/2}M
      \end{array}
    \right]^{-1} \left[ 
      \begin{array}{cc}
        \bar{K} & M \\
        M & -\eta \bar{K}  
      \end{array}
    \right] \right) = O(1).
   \end{align}
\end{theorem}
For sake of comparison, we also consider block triangular preconditioners of the
form
% \begin{align}
%   \mathcal{B} = 
%   \begin{pmatrix}
%     X_{11} & \\
%     X_{21} & X_{22}
%   \end{pmatrix}
% \end{align}
used in Bosch et. al. \cite{Bosch2014}. In the following, we briefly describe
this preconditioner in our notation.
\subsubsection{\label{sec:precon} Preconditioner II}
In Bosch et. al. \cite{Bosch2014}, a preconditioner is proposed in the framework
of a semi-smooth Newton method combined with Moreau-Yosida regularization for
the same problem. However, the preconditioner was constructed for a linear
system which is different from the one we considered here \eqref{eqn:rls}.  The
preconditioner proposed in \cite{Bosch2014} has the following block lower
triangular form
\begin{align} \label{eqn:block_lower}
  \mathcal{B} =
  \begin{pmatrix}
    \bar{K}                      & 0                     \\
    M                      & -S
  \end{pmatrix},
\end{align}
where $S = C + M \bar{K}^{-1}M^T$ is the Schur complement. Note
that such preconditioners are also called inexact or preconditioned Uzawa
preconditioners for linear saddle point problems.  Both $\bar{K}$ and $S$ are
invertible \cite{Graeser2009b}. Hence by block $2 \times 2$ inversion formula we
have
\begin{align}
\mathcal{B}^{-1} = 
  \begin{pmatrix}
    \bar{K}                      & 0                     \\
    M                            & -S
  \end{pmatrix}^{-1} = 
  \begin{pmatrix}
    \bar{K}^{-1}                 & 0                     \\
    S^{-1}M^T \bar{K}^{-1} & -S^{-1}
  \end{pmatrix}.
\end{align} 
Let $\tilde{S}$ be any approximation of Schur
complement $S$ in $\mathcal{B}$ in \eqref{eqn:block_lower}, then the new
preconditioner $\hat{\mathcal{B}},$ and the corresponding preconditioned
operator $\hat{\mathcal{B}}^{-1}\mathcal{A}$ is given as follows
\begin{align}
  \hat{\mathcal{B}} = \begin{pmatrix}
    \bar{K}                      & 0                     \\
    M                            & -\tilde{S}
  \end{pmatrix}, \quad \hat{\mathcal{B}}^{-1}\mathcal{A}
  = \begin{pmatrix}
    I                            & \bar{K}^{-1}M^T \\
    0                            & \tilde{S}^{-1}S
  \end{pmatrix}. \label{eqn:BinvA}
\end{align}
In this paper we choose $\tilde{S}$ as follows
\begin{align}
 \tilde{S} = S_1 \bar{K}^{-1}S_2 = (M + \sqrt \eta \bar{K})\bar{K}^{-1}(M + \sqrt \eta \bar{K})
\end{align}

We note the following trivial result.
\begin{fact} \label{thm:eigleft} Let $\mathcal{B}$ be defined as in
  \eqref{eqn:BinvA}, then there are $|\mathcal{N}_h|$ eigenvalues of
  $\mathcal{B}^{-1}\mathcal{A}$ equal to one, and the rest are the eigenvalues
  of the preconditioned Schur complement $\tilde{S}^{-1}S.$
\end{fact}
\begin{remark}
  When using GMRES \cite{Saad2003}, right preconditioning is preferred.  Similar
  result as for the left preconditioner above Theorem \ref{thm:eigleft} holds.
\end{remark}
Let $x=[x_1, x_2], b = [b_1, b_2].$ The preconditioned system
$\mathcal{B}^{-1}\mathcal{A}x = \mathcal{B}^{-1}b$ is given as follows
\begin{align}
\begin{pmatrix}
    I                            & \hat{A}^{-1}\hat{B}^T \\
    0                            & \hat{S}^{-1}S
  \end{pmatrix}
  \begin{pmatrix}
    x_1                                                  \\ x_2
  \end{pmatrix} =
    \begin{pmatrix}
    \hat{A}^{-1}                 & 0                     \\
    S^{-1}\hat{B}^T \hat{A}^{-1} & -S^{-1}
  \end{pmatrix}
  \begin{pmatrix}
    b_1                                                  \\ b_2
  \end{pmatrix}
\end{align}
from which we obtain the following set of equations
\begin{align}
  x_1 + \hat{A}^{-1}\hat{B}^T x_2 = \hat{A}^{-1}b_1, \quad \hat{S}^{-1}Sx_2 =
  S^{-1}(\hat{B}^T \hat{A}^{-1} b_1 - b_2).
\end{align}
\begin{alg}\label{alg:solve}
Objective: Solve $\mathcal{B}^{-1}\mathcal{A}x = \mathcal{B}^{-1}b$
  \begin{enumerate}
  \item Solve for $x_2:$ $\hat{S}^{-1}Sx_2 =
    \hat{S}^{-1}(\hat{B}^T\hat{A}^{-1}b_1 - b_2)$
  \item Set $x_1 = \hat{A}^{-1}(b_1 - \hat{B}^Tx_2)$
  \end{enumerate}
\end{alg}
Here if Krylov subspace method is used to solve for $x_1$, then matrix vector
product with $S$ and a solve with $\hat{S}$ is needed. However, when the problem
size, i.e., $|\mathcal{N}_h|$ is large, it won't be feasible to do exact solve
with $\hat{A},$ and we need to solve it inexactly, for example, using algebraic
multigrid methods. In the later case, the decoupling of $x_1$ and $x_2$ as in
Algorithm \ref{alg:solve} is not possible; then the preconditioned Schur complement $\tilde{S}^{-1}S$ is not
symmetric, so we use GMRES in Saad \cite[p. 269]{Saad2003} that allows
nonsymmetric preconditioners.

\section{\label{sec:numexp} Numerical Experiments}
All the experiments were performed in double precision arithmetic in MATLAB. The
Krylov solver used was GMRES with subspace dimension of 200, and maximum number
of iterations allowed was 300. The iteration was stopped as soon as the relative
residual was below the tolerance of $10^{-7}.$

\subsection{\label{sec:spec} Spectrum Analysis}
 
We consider two samples of active set configurations that occur when a square
region evolves as shown in figures \ref{fig:square_pink} and \ref{fig:circle_pink}. The
region between the two squares and the circles is the interface between two bulk
phases taking values +1 and -1; initially we chose random values between -0.3
and 0.5 in the interface region. The width of the interface is kept to be 10
times the chosen mesh size. The time step $\tau$ is chosen to be equal to
$\epsilon.$ We compare various mesh sizes leading to number grid points up to
just above 1 million, and compare various values of epsilon for each mesh
sizes. We observe that the number of iterations remain independent of the mesh
size, however it depends on $\epsilon.$ But we observe that for a fixed epsilon,
with finer mesh, the number of iterations actually decrease significantly. For
example the number of iterations for $h=2^{-7}, \epsilon=10^{-6}$ is 84 but the
number of iterations for $h=2^{-10}, \epsilon=10^{-6}$ is 38, a reduction of 46
iterations!  It seems that finer mesh size makes the preconditioner more
efficient. We also observe that the time to solve is proportional to number of
iterations; the inexact solve for the (1,1) block remains optimal because the
(1,1) block is essentially Laplacian for which AMG remains very efficient.
\begin{figure}[tbp]
  \centering  
  \subfigure{\label{fig:square_pink}\includegraphics[width=7.3cm, height=7.3cm]{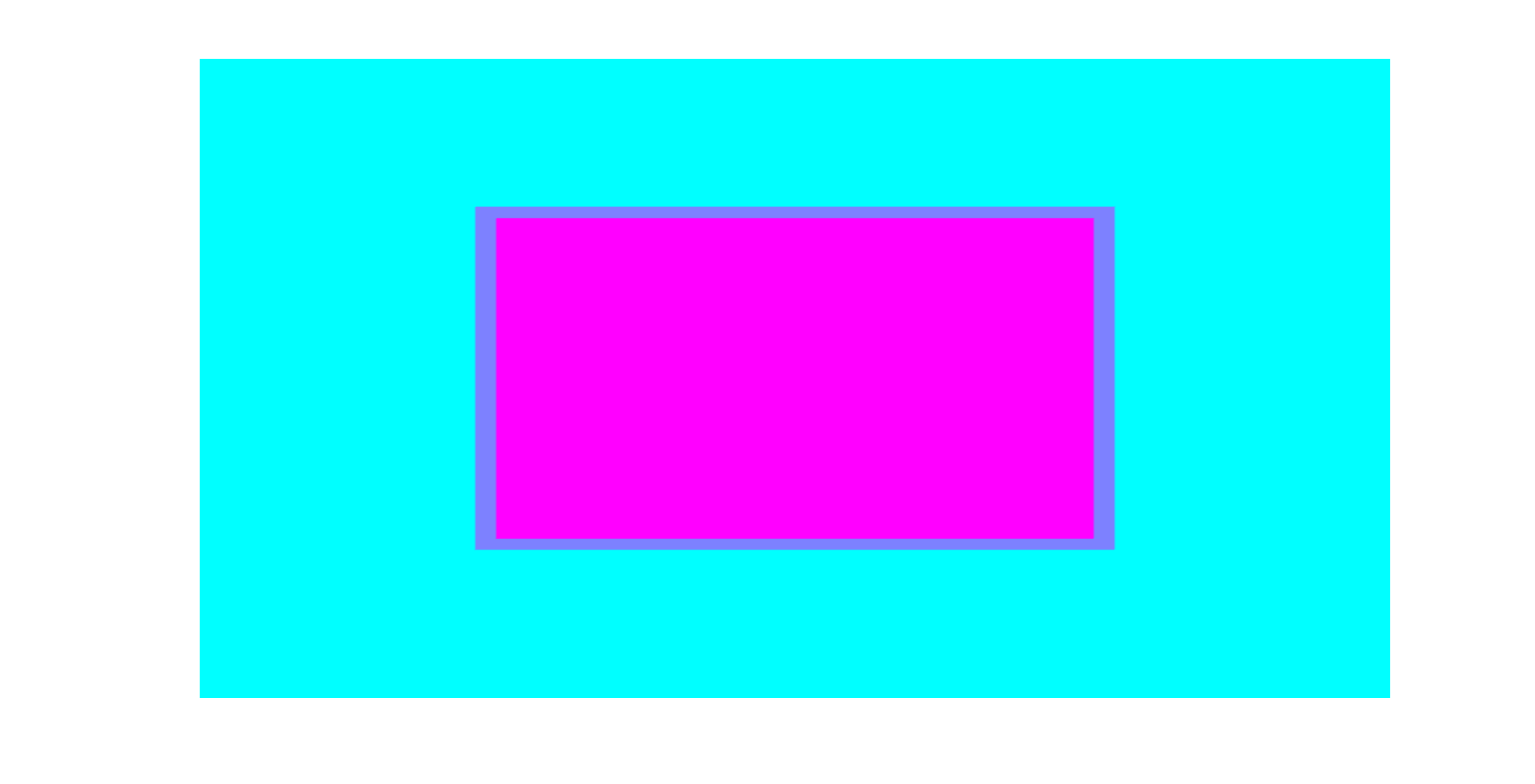}}
  \hfill  
  \subfigure{\label{fig:circle_pink}\includegraphics[width=7.3cm, height=7.3cm]{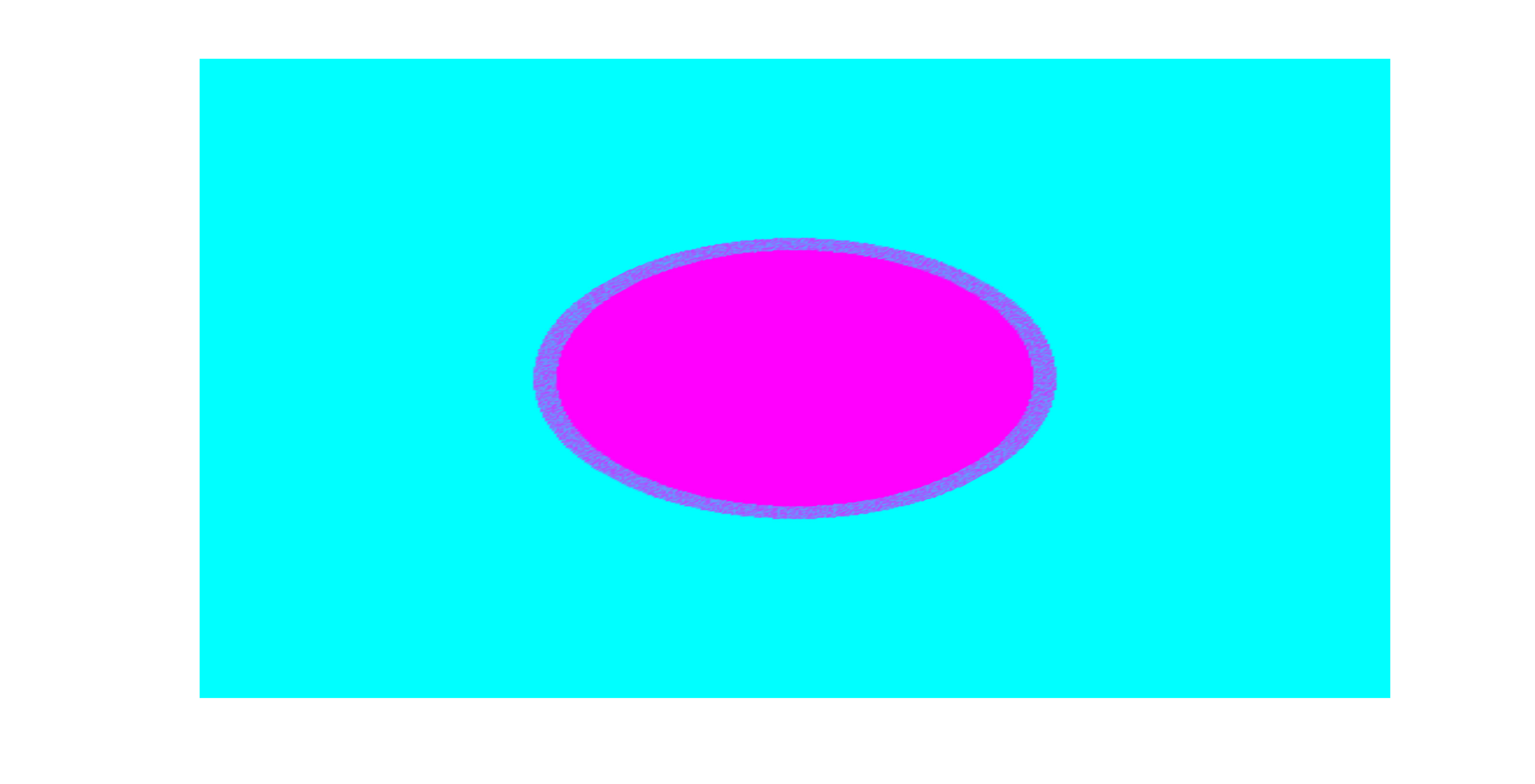}}
  \caption{Active set configurations: Square and Circle}
\end{figure}
 
\begin{table}
  \centering
  \begin{tabular}{cc|cc|cc}
    \hline 
                  &             & \multicolumn{2}{|c|}{square} & \multicolumn{2}{|c}{circle} \\
    \hline
    $h$           & $\epsilon$  & its                          & time   & its & time         \\
    \hline
    $2^{-9}$      & e-2         & 10                           & 49.24  & 10  & 33.04        \\
                  & e-3         & 11                           & 54.25  & 10  & 36.63        \\
                  & e-4         & 36                           & 173.67 & 34  & 117.32       \\
                  & e-5         & 101                          & 418.05 & 89  & 334.28       \\
    \hline
    $2^{-10}$     & e-2         & 10                           & 162.05 & 9   & 149.18       \\
                  & e-3         & 12                           & 175.00 & 10  & 167.47       \\
                  & e-4         & 22                           & 313.02 & 19  & 276.98       \\
                  & e-5         & 73                           & 1113.2 & 65  & 954.67       \\
    \hline
  \end{tabular} \quad
  \begin{tabular}{cc|cc|cc}
    \hline 
    % \cline{1-2} & \cline{3-4} & \cline{5-6}                                                \\
                  &             & \multicolumn{2}{|c|}{square} & \multicolumn{2}{|c}{circle} \\
    %             &             & \cline{3-4}                  & \cline{5-6}                 \\
    \hline
    $h$           & $\epsilon$  & its                          & time   & its & time         \\
    \hline
    $2^{-9}$      & e-2         & 11                           & 41.36  & 10  & 46.19        \\
                  & e-3         & 19                           & 71.29  & 17  & 81.93        \\
                  & e-4         & 22                           & 90.86  & 22  & 114.12       \\
                  & e-5         & 20                           & 102.65 & 20  & 108.45       \\
    \hline
    $2^{-10}$     & e-2         & 11                           & 212.94 & 11  & 232.30       \\
                  & e-3         & 14                           & 315.77 & 13  & 281.88       \\
                  & e-4         & 27                           & 596.06 & 24  & 703.70       \\
                  & e-5         & 19                           & 493.07 & 19  & 469.09       \\
    \hline
  \end{tabular}
  \caption{\label{table}Iteration count for various $\epsilon$ and $h;$ Left: Preconditioner I; Right:
    Preconditioner II}
\end{table}

\section{Conclusion}
For the solution of large scale optimization problem corresponding to
Cahn-Hilliard problem with obstacle problem, we proposed an efficient
preconditioning strategy that requires two elliptic solves. In our initial
experiments up to over million unknowns, the preconditioner remains mesh
independent.  Although, for coarser mesh there seems to be strong dependence on
the epsilon, but as the mesh becomes finer, we observe a significant reduction
in iteration count, thus making the preconditioner effective and useful on finer
meshes.  It is likely that the the iteration count further decreases on finer
meshes.

\bibliography{mybib}
\bibliographystyle{plain}

\end{document}